\documentclass[12pt, reqno]{amsart}
\usepackage{amssymb,a4}
\usepackage{amsmath,amsfonts,amssymb,amsthm}
\usepackage{mathrsfs}
\usepackage{color}
\usepackage{colordvi}

\newcommand{\on}{\operatorname}
\newcommand{\cal}{\mathcal}
\newcommand{\f}{\mathfrak}

\newcommand{\al}{\alpha}

\def\wt{{\rm wt}}

\def\de{\delta}
\def\be{\beta}

\def\C{{\mathbb C}}

\def\Z{{\mathbb Z}}

\def\N{{\mathbb N}}
\def\1{{\bf 1}}
\def\l{\lambda}

\def \End{{\rm End}}

\def \<{\langle}
\def \>{\rangle}
\def \w{\omega}
\def \fg{\frak{g}}
\def \n{\frak{n}}
\def \wg{{\widehat{\frak{g}}}}
\def \wh{{\widehat{\frak{h}}}}

\def \sl{\frak{sl}}
\def \mw{\mathcal{W}}

\def \wD{\widehat{\Delta}}
\def \Cl{\mathcal  Cl}
\def \wW{\widetilde{W}}
\def \h{\mathfrak{h}}
\def \l{\lambda}
\def \w{\omega}
\def \F{\mathcal F}
\def \ww {\widetilde{\omega}}
\def \W{\mathcal{W}}
\def \wv{\widetilde{V}}
\numberwithin{equation}{section}

\newtheorem{theorem}{Theorem}[section]
\newtheorem{prop}[theorem]{Proposition}
\newtheorem{lem}[theorem]{Lemma}

\newtheorem{remark}[theorem]{Remark}
\newtheorem{conjecture}[theorem]{Conjecture}
\theoremstyle{definition}
\newtheorem{defn}[theorem]{Definition}

\begin{document}
\title[Coset Vertex Operator Algebras and $\W$-Algebras]{Coset vertex operator algebras and $\W$-algebras of $A$-type}

\author{Tomoyuki Arakawa}\thanks{Arakawa is supported by JSPS KAKENHI Grants (25287004 and 26610006) }
\address[Arakawa]{ Research Institute for Mathematical Sciences, Kyoto University, Kyoto, 606-8502, Japan}
\email{arakawa@kurims.kyoto-u.ac.jp}

\author{ Cuipo  Jiang} \thanks{Jiang is supported by CNSF grants (11371245 and  11531004)}
\address[Jiang]{School of Mathematical Sciences,  Shanghai Jiao Tong University, Shanghai, 200240, China }
\email{cpjiang@sjtu.edu.cn}

\begin{abstract}We give an explicit   description for the weight three generator of the coset vertex operator algebra $C_{L_{\widehat{\sl_{n}}}(l,0)\otimes L_{\widehat{\sl_{n}}}(1,0)}(L_{\widehat{\sl_{n}}}(l+1,0))$, for $n\geq 2, l\geq 1$. Furthermore, we prove that  the commutant $C_{L_{\widehat{\sl_{3}}}(l,0)\otimes L_{\widehat{\sl_{3}}}(1,0)}(L_{\widehat{\sl_{3}}}(l+1,0))$ is isomorphic to the $\W$-algebra $\W_{-3+\frac{l+3}{l+4}}(\sl_3)$, which confirms   the conjecture for the $\sl_3$ case
that $C_{L_{\widehat{\frak g}}(l,0)\otimes L_{\widehat{\frak g}}(1,0)}(L_{\widehat{\frak g}}(l+1,0))$ is isomorphic to
$\W_{-h+\frac{l+h}{l+h+1}}(\frak g)$ for  simply-laced Lie algebras ${\frak g}$ with its Coxeter number $h$  for a  positive integer $l$.
 \end{abstract}
\subjclass[2010]{17B69}

\maketitle
\section{Introduction}
\subsection{}  Given a vertex operator algebra $V$ and a vertex operator subalgebra $U\subseteq V$, $C_V(U)$ which  is called the commutant of $U$ in $V$ or coset construction, is the subalgebra of $V$ which commutes with $U$. The coset vertex algebra construction, initiated in [GKO], was introduced by Frenkel and Zhu in [FZ]. Coset construction is one of the major ways to construct new vertex operator algebras from given ones.  It is widely believed that the commutant $C_V(U)$ inherits properties from $V$ and $U$ in many ways. In particular it is expected that $C_V(U)$ is rational if both $V$ and $U$ are rational. No general  results of this kind have been  known so far.  Nevertheless, many interesting examples, especially coset vertex operator algebras related to affine vertex operator algebras,  have been  extensively studied both in the physics and mathematics literature  \cite{AP}, \cite{ALY1}, \cite{ALY2}, \cite{BEHHH}, \cite{BFH},  \cite{ChL}, \cite{CL}, \cite{DJX}, \cite{DLY}, \cite{DLWY}, \cite{DW1}, \cite{DW2}, \cite{GQ}, \cite{JL1}, \cite{JL2}, \cite{LS}, \cite{LY}, etc.

Let $\wg$  and $\wh$ be the  affine Kac-Moody Lie algebras associated to a complex simple Lie algebra $\fg$ and its Cartan subalgebra $\h$, respectively. For $k\in \C$ such that $k\neq -h^{\vee}$, where $h^{\vee}$ is the dual Coxter number of $\fg$, let $L_{\wg}(k,0)$ and $L_{\wh}(k,0)$ be the associated simple vertex operator algebras with level $k$, respectively \cite{FZ}, \cite{LL}.  The commutant of the Heisenberg vertex operator algebra $L_{\wh}(k,0)$ in $L_{\wg}(k,0)$, denoted by $ K(\fg, k)$, is the so called  parafermion vertex operator algebra \cite{ZF}. Parafermion vertex operator algebras have been studied extensively \cite{ALY1}-\cite{ALY2}, \cite{CGT}, \cite{DLWY}, \cite{DLY},   \cite{DW1}-\cite{DW2}, \cite{JL1}, \cite{LY},  etc. Among other things,  for a positive integer $k$,  the $C_2$-cofiniteness of $K(\fg,k)$   was established in \cite{ALY1} and \cite{DW2}, the generators of $K(\fg,k)$ were given in \cite{DW1}, and the rationalities of $K(\sl_2,k)$ and $K(\sl_k,2)$ were established in \cite{ALY2} and \cite{JL1}-\cite{JL2}, respectively.

Let $L_{\wg}(1,0)$ be the simple affine vertex operator algebra  associated to the simple Lie algebra $\fg$ with level $1$. For $l\in\Z_{\geq 2}$, $L_{\wg}(1,0)^{\otimes l}$ has a tensor product
vertex operator algebra structure with $L_{\wg}(l,0)$ being a vertex operator subalgebra.  The commutant
 of $L_{\wg}(l,0)$ in $L_{\wg}(1,0)^{\otimes l}$ is  a simple
vertex operator subalgebra of $L_{\wg}(1,0)^{\otimes l}$. It was proved in \cite{JL2} and in \cite{L} independently that $C_{L_{\widehat{\frak{sl}_n}}(1,0)^{\otimes l}}(L_{\widehat{\frak{sl}_n}}(l,0)) \cong
 K(\sl_l,n)$ as vertex operator algebras, presenting a version of level-rank duality.
 The classification of irreducible modules and the rationality of $C_{L_{\widehat{\frak{sl}_2}}(1,0)^{\otimes l}}(L_{\widehat{\frak{sl}_2}}(l,0))$ were established in \cite{JL1}. Then by the level-rank duality, the parafermion vertex operator algebra $K(\sl_l,2)$ is rational for any $l\in\Z_{\geq 2}$.

 More generally, given a sequence of positive integers
$\underline{\ell}=(l_1, \cdots l_s)$, the tensor product vertex operator algebra
$L_{\wg}(\underline{\ell}, 0)= L_{\wg}(l_1, 0)\otimes L_{\wg}(l_2, 0)\otimes
\cdots \otimes L_{\wg}(l_s,0)$
contains a vertex operator subalgebra isomorphic to
$ L_\wg(|\underline{\ell}|,0)$ with $|\underline{\ell}|=l_1+\cdots+l_s$.  On the other hand, the sequence
$\underline{\ell}$ defines a Levi subalgebra $\frak{l}_{\underline{\ell}}$ of
$\frak{sl}_{|\underline{\ell}|}$. Denote by $L_{\widehat{\frak{l}_{\underline{\ell}}}}(n,0)$ the vertex operator subalgebra of
$ L_{\widehat{\frak{sl}_{|\underline{\ell}|}}}(n, 0)$ generated by
$\frak{l}_{\underline{\ell}}$.
Set $K(\frak{sl}_{|\underline{\ell}|}, \frak{l}_{\underline{\ell}}, n)=
C_{L_{\widehat{\frak{sl}_{|\underline{\ell}|}}}(n, 0)}(L_{\widehat{\frak{l}_{\underline{\ell}}}}(n, 0)).$
It was established in \cite{JL2} that $C_{L_{\widehat{\frak{sl}_n}}(\underline{\ell}, 0)}(L_{\widehat{\frak{sl}_n}}(|\underline{\ell}|, 0))\cong K(\frak{sl}_{|\underline{\ell}|}, \frak{l}_{\underline{\ell}}, n)$,  which is a more general version of rank-level duality. In particular, we have
 $C_{L_{\widehat{\f{sl}_n}}(l,0)\otimes  L_{\widehat{\f{sl}_n}}(1,0)}( L_{\widehat{\f{sl}_n}}(l+1,0))\cong
 C_{L_{\widehat{\sl_{l+1}}}(n,0)}(L_{\widehat{\sl_{l}}}(n,0)\otimes L_{\widehat{\frak{h_l}}}(n,0))$.

(The principal) $\W$-algebras are one-parameter families of vertex algebras associated to simple Lie algebras. First example of a $\W$-algebra was introduced by Zamolodchikov  \cite{Za} in an attempt  to classify extended conformal algebras with two generating fields. Since then there have been several approaches to the construction of $\W$-algebras \cite{BFKNNRV},  \cite{FL}, \cite{FZa}, \cite{Lu}, \cite{LF}, \cite{BG}, \cite{BBSS2}, \cite{Bou}, \cite{BP}, \cite{Mi}, \cite{BBSS1}, \cite{FFr3}, \cite{FFr4}.

   For $k\in\C$, let $V_{\wg}(k,0)$ be the universal affine vertex operator algebra associated to the simple Lie algebra $\fg$ with level $k$. The associated (principal) affine  $\W$-algebra $\W^k(\fg)$, which is a vertex algebra, can be realized as the cohomology of the BRST complex of the quantum Drinfeld-Sokolov reduction \cite{FFr3}, \cite{FFr4}. From this point of view, $\W$-algebras have been studied  deeply and extensively, and many remarkable results have been achieved \cite{BO}, \cite{FFr1}-\cite{FFr4}, \cite{Fe}, \cite{F}, \cite{KWa1}-\cite{KWa4}, \cite{FKW}, \cite{FBZ}, \cite{Ar1}-\cite{Ar4}, \cite{AM}, etc.  Among other things, it was established in \cite{Ar2} that the character of each irreducible highest weight representation of $\W^k(\fg)$ is completely determined by that of the corresponding irreducible highest weight representations of the affine Lie algebra $\wg$. The $C_2$-cofiniteness  and the rationality of  the minimal series principal $\W$-algebras were established in \cite{Ar4} and \cite{Ar5}, respectively.

    Recall that \cite{FKW} a rational number $k$ with the denominator $u\in \N$ is called  principal admissible if
   $$
   u(k+h^{\vee})\geq h^{\vee}, \ \ (u,r^{\vee})=1,
   $$
   where $h$ is the Coxeter number of $\fg$, $h^{\vee}$ is the dual Coxeter number of $\fg$, and  $r^{\vee}$ is the maximal number of the edges in the Dykin diagram of $\fg$.
   A principal admissible number $k$ is called non-degenerate
   if the denominator $u$ is equal or greater than the Coxeter number $h$.
    For a non-degenerate admissible principal number $k$,
    denote by $\W_k(\fg)$ the simple quotient of $\W^k(\fg)$, which are also called minimal series principal $\W$-algebras \cite{FKW}, \cite{Ar4}, \cite{Ar5}.

   One conjecture \cite{FKW}, \cite{BS}  about minimal series principal $\W$-algebras asserts that  $\W_k(\fg)$ is isomorphic to the commutant $C_{L_{\wg}(p,0)\otimes L_{\wg}(1,0)}(L_{\wg}(p+1,0))$ of
$L_{\wg}(p+1,0)$ in $L_{\wg}(p,0)\otimes L_{\wg}(1,0)$ for  $\fg$ being simply-laced and $k=-h^{\vee}+\frac{p+h^{\vee}}{p+h^{\vee}+1}$, $p\in\Z_{+}$. This conjecture comes partially from the fact that $\W_k(\fg)$ and $C_{L_{\wg}(p,0)\otimes L_{\wg}(1,0)}(L_{\wg}(p+1,0)$ share the same normalized characters \cite{KWa2}, \cite{FKW}.
If $\fg=\sl_2(\C)$, then  $\W_k(\fg)$ is the simple Virasoro vertex operator algebra $L(c_p,0)$, where $c(k)=1-\frac{1}{(p+2)(p+3)}$. For this case, the conjecture follows from the Goddard-Kent-Olive construction \cite{GKO}. If $\fg=\sl_n(\C) (n\geq 2)$ and $p=1$, it was proved in \cite{ALY2} that
$\W_k(\fg)$ is isomorphic to the parafermion vertex operator algebra $K(\sl_2,n)$, with which the rationality of $K(\sl_2,n)$ is established. Then the conjecture in this case follows from the fact that $K(\sl_2,n)\cong C_{L_{\widehat{\frak{sl}_{n}}}(1,0)^{\otimes 2}}(L_{\widehat{\sl_{n}}}(2,0))$ \cite{L}, \cite{JL2}. For general $l\in\Z_{\geq 1}$  and $n\in\Z_{\geq 2}$, by the level-rank duality and the reciprocity law given in \cite{JL2},
$C_{L_{\widehat{\f{sl}_n}}(l,0)\otimes  L_{\widehat{\f{sl}_n}}(1,0)}( L_{\widehat{\f{sl}}_n}(l+1,0))\cong
 C_{K(\f{sl}_{l+1}, n)}(K(\f{sl}_{l}, n)).$ In this paper, we first give explicitly a generator of $C_{K(\f{sl}_{l+1}, n)}(K(\f{sl}_{l}, n))$ of weight three. Then we prove that the conjecture is true for the case that $\fg=\sl_3(\C)$.

 The paper is organized as follows. In Section 2, we briefly review some basics on vertex operator algebras. In Section 3, we recall some notations and facts about  principal affine $\W$-algebras. In Section 4  we review  parafermion vertex operator algebras and the level-rank duality for tensor power decompositions of rational vertex operator algebras of type $A$. In Section 5, we study the coset vertex operator algebra  $C_{K(\f{sl}_{l+1}, n)}(K(\f{sl}_{l}, n))$. The main results of this paper are stated in this part.

\section{Preliminaries}
In this section, we recall some notations and basic facts on vertex operator algebras  \cite{B},
\cite{FLM}, \cite{LL}, \cite{Z}, \cite{DLM1}.

\begin{defn}
A vertex operator algebra $V=(V,Y,{\bf 1},\omega)$ is defined as
follows

(1) $V=\bigoplus_{n\in{\Z}}V_n$ is a $\Z$-graded vector space over $\C$
  equipped with a linear map $Y$:
$$\ \begin{array}{l}
V\to (\End\,V)[[z,z^{-1}]]\\
v\mapsto\displaystyle{ Y(v,z)=\sum_{n\in Z}v_nz^{-n-1}\ \ \ (v_n\in
\End\,V)},
\end{array}$$
such that $\dim V_n<\infty,$ $V_m=0$ if $m<<0$, and for any $u,v\in V$, $v_nu=0$ for sufficiently large $n$.

(2) There exist two distinguished two vectors: the vacuum ${\bf 1}\in V_0$, and the conformal vector $\omega\in V_2$ such that
 $$Y({\bf 1},z)=id, \ \ {\rm
lim}_{x\rightarrow 0}Y(u,x){\bf 1}=u,$$
$$Y(\omega,z)=\sum_{n\in \Z}L(n)z^{-n-2},$$
$$[L(m),L(n)]=(m-n)L(m+n)+\frac{m^3-m}{12}\delta_{m+n,0}c.$$
The complex number $c$ is called the central charge of $V$. Moreover,  $L(0)=n$ on $V_n$, and
for $u\in V$,
$$[L(-1),Y(u,z)]=\frac{d}{dz}Y(u,z).$$

(3) For any $u,v\in V$ there exists $n\geq 0$ such that
$$(z_1-z_2)^nY(u,z_1)Y(v,z_2)=(z_1-z_2)^nY(v,z_2)Y(u,z_1).$$
\end{defn}

\begin{defn} Let $(V, {\bf 1}, \omega, Y)$ be a vertex operator algebra. A weak $V$-module is a vector space $M$ equipped
with a linear map
$$
\begin{array}{ll}
Y_M: & V \rightarrow {\rm End}(M)[[z,z^{-1}]]\\
 & v \mapsto Y_M(v,z)=\sum_{n \in \Z}v_n z^{-n-1},\ \ v_n \in {\rm End}M
\end{array}
$$
satisfying the following:

(1) $v_nw=0$ for $n>>0$ where $v \in V$ and $w \in M$,

(2) $Y_M( {\textbf 1},z)=\on{id}_M$,

(3) The Jacobi identity holds:
\begin{eqnarray}
& &z_0^{-1}\de \left({z_1 - z_2 \over
z_0}\right)Y_M(u,z_1)Y_M(v,z_2)-
z_0^{-1} \de \left({z_2- z_1 \over -z_0}\right)Y_M(v,z_2)Y_M(u,z_1) \nonumber \\
& &\ \ \ \ \ \ \ \ \ \ =z_2^{-1} \de \left({z_1- z_0 \over
z_2}\right)Y_M(Y(u,z_0)v,z_2).
\end{eqnarray}
\end{defn}


\begin{defn}
An admissible $V$-module is a weak $V$-module  which carries a
$\Z_+$-grading $M=\bigoplus_{n \in \Z_+} M(n)$, such that if $v \in
V_r$ then $v_m M(n) \subseteq M(n+r-m-1).$
\end{defn}

\begin{defn}
An ordinary $V$-module is a weak $V$-module which carries a
$\C$-grading $M=\bigoplus_{\l \in \C} M_{\l}$, such that:

1) $\dim(M_{\l})< \infty,$

2) $M_{\l+n}=0$ for fixed $\l$ and $n<<0,$

3) $L(0)w=\l w=\wt(w) w$ for $w \in M_{\l}$, where $L(0)$ is the
component operator of $Y_M(\omega,z)=\sum_{n\in\Z}L(n)z^{-n-2}.$
\end{defn}

It is easy to see that an ordinary $V$-module is an admissible one. If $W$  is an
ordinary $V$-module, we simply call $W$ a $V$-module.

We call a vertex operator algebra $C_2$-cofinite if $V/C_2(V)$ is finite-dimensional  where $C_2(V)=\<u_{-2}v|u,v\in V\>$ \cite{Z}. A vertex operator algebra is called rational if the admissible module
category is semisimple \cite{Z}, \cite{DLM1}. We have the following result from \cite{Z},
\cite{DLM1}, and \cite{ABD}.

\begin{theorem}\label{tt2.1}
 If $V$ is a vertex operator algebra satisfying the $C_{2}$-cofinite
property, $V$ has only finitely many irreducible admissible modules
up to isomorphism. The rationality of of $V$ also implies the same result.
\end{theorem}

Let $(V, Y, {\bf 1}, \omega)$ be a vertex operator algebra and $(U, Y, {\bf 1}, \omega')$ a vertex operator subalgebra of $V$. Set
$$
C_{V}(U)=\{v\in V| [Y(u,z_1), Y(v,z_2)]=0, u\in U \}.
$$
Recall from \cite{FZ} and \cite{LL} that if
$$
\omega'\in U\cap V_2
$$
and
$$
L(1)\omega'=0,
$$
then
$$
C_{V}(U)=\{v\in V| u_{m}v=0, u\in U, m\geq 0 \}
$$
is a vertex operator subalgebra of $V$ with the conformal vector $\omega-\omega'$. We shall write
$$
Y(\omega',z)=\sum\limits_{n\in\Z}L'(n)z^{-n-2},
$$
where we view the operators $L'(n)$ as  acting on $V$. We have the following result from \cite{LL}.
\begin{prop} Let $(V, Y, {\bf 1}, \omega)$ be a vertex operator algebra and $(U, Y, {\bf 1}, \omega')$ a vertex operator subalgebra of $V$. Then
$$C_V(U)=\operatorname{Ker}_VL'(-1).
$$
\end{prop}

\setcounter{equation}{0}

\section{$\W$-algebras for principle case}

In this section, we recall some notations and facts on the $\W$ algebras. ${\mathcal W}$-algebras may be defined in several ways. The definition here    was given by using the quantum Drinfeld-Sokolov reduction  \cite{FFr1}, \cite{FFr2}, \cite{FKW}, \cite{KRW}, \cite{FBZ}, \cite{Ar2}.  Throughout this section, $k$ is a complex number with no restriction unless otherwise stated.

\subsection{} Let $\fg$ be a complex simple Lie algebra of rank $l$ . Let $(\cdot|\cdot)$ be the normalized non-degenerate bilinear form on $\fg$, that is, $(\cdot|\cdot)=\frac{1}{2h^{\vee}}\cdot $Killing form, where $h^{\vee}$ is the dual Coxeter number of $\fg$.

Let $e$ be a principal nilpotent element of $\fg$ so that $\dim\fg^{e}=l$. By the Jacobson-Morozov
theorem, there exists an $\sl_2$-triple $\{e,f,h_0\}$ associated to $e$ satisfying
$$
[h_0, e]=2e, \  [h_0, f]=-2f, \  [e,f]=h_0.
$$
Set
$$
\fg_j=\{x\in\fg|[h_0, x]=2jx\},  \  {\rm for}  \  j\in\Z.
$$
This gives a triangular decomposition
 $$
\fg=\n_+\oplus\h\oplus\n_{-},
$$
where
$$
\h=\fg_0, \  \n_+=\bigoplus_{j\geq 1}\fg_j, \   \n_-=\bigoplus_{j\geq 1}\fg_{-j}.
$$
Denote by $\Delta_+\subset \mathfrak{h}^*$ the set of positive roots of $\fg$ and by $\{\al_1,\cdots, \al_l\}$ the subset of simple roots. Let $\n_+^*$ be the dual of $\n_+$. Define $\bar{\chi}_+\in \n_+^*$ by
$$
\bar{\chi}_+(x)=(f|x), \  \  {\rm for}  \ x\in\n_+.
$$
Then $\bar{\chi}_+$ is a character of $n_+$, that is, $\bar{\chi}_{+}([n_+, n_+])=0$. Let $\wg$ be the non-twisted affine Lie algebra associated to $\fg$ (see \cite{K} for details). That is,
$$
\wg=\fg\otimes \C [t,t^{-1}]\oplus\C K\oplus\C D,
$$
with the commutation relations
$$
[X(m), Y(n)]=[X, Y](m+n)+m\delta_{m+n,0}(X|Y)K,
$$
$$
[D, X(m)]=mX(m), \ [K, \fg]=0
$$
for $X,Y\in\fg$, $m,n\in\Z$, where $X(n)=X\otimes t^n$. The invariant symmetric  bilinear form $(\cdot|\cdot)$ is extended from $\fg$ to $\wg$ as follows:
$$
\wh=\h\oplus\C K\oplus \C D,
$$
$$
\wg_+=\n_+\otimes \C[t]\oplus (\n_-\oplus \h)\otimes \C[t]t,
$$
$$
\wg_-=\n_-\otimes \C[t^{-1}]\oplus(\n_++\h)\C[t^{-1}]t^{-1}.
$$
Let $$\wh^*=\h^*\oplus\C\Lambda_0\oplus\C\delta$$
be the dual of $\wh$, where $\Lambda_0$ and $\delta$ are the dual elements of $K$ and $D$, respectively.  For $\lambda\in\wh^*$, the number $\lambda(K)$ is called the level of $\lambda$. Let
$\widehat{\Delta}$ be the set of roots of $\wg$, $\Delta_+$ the set of positive roots, and $\widehat{\Delta}_-=-\wD_+$. Denote by $\wD^{re}$ and $\wD^{im}$ the set of real roots and the set of imaginary roots, respectively. Then
$$
\wD^{im}=\{n\delta \ | \ n\in\Z\}, \  \wD_+^{re}=\{\al+n\delta, -\al+m\delta, \ | \ \al\in\Delta_+, n\in\Z_{\geq 0}, m\in\Z_{\geq 1} \}.
$$
Let $V_{\wg}(k,0)$ be the universal vertex operator algebra associated to $\fg$ with level $k$ (see \cite{FZ}, \cite{LL}, \cite{K} for details).  Let ${\mathcal Cl}$ be the superalgebra generated by odd generators: $\psi_{\alpha}(n)$, $\al\in\Delta, n\in\Z$ with the following super Lie relations:
$$
[\psi_{\al}(m), \psi_{\be}(n)]_{+}=\delta_{\al+\be,0}\delta_{m+n,0},
$$
for $\al,\be\in\Delta$, $m,n\in\Z$.  Here $\psi_{\al}$ is regarded as the element of ${\mathcal Cl}$ corresponding to the root vector $e_{\al}(n)\in\wg_{\al}$.  Let $\F$ be the irreducible $\Cl$-module generated by the cycle vector $\1$ such that
$$
\psi_{\al}(n)\1=0, \  {\rm if} \  \al+n\delta\in\wD_{+}^{re}.
$$
$\F$ is naturally a vertex operator superalgebra with  the vacuum vector $\1$ and the fields defined by
$$
Y(\psi_{\al}(-1)\1,z)=\psi_{\al}(z):=\sum\limits_{n\in\Z}\psi_{\al}(n)z^{-n-1}, \  {\rm for} \  \al\in\Delta_{+},
$$
$$
Y(\psi_{\al}(0)\1,z)=\psi_{\al}(z):=\sum\limits_{n\in\Z}\psi_{\al}(n)z^{-n}, \  {\rm for} \  \al\in\Delta_{-}.
$$
The conformal vector is chosen as $\omega=\sum\limits_{\al\in\Delta_+}\psi_{-\al}(-1)\psi_{\al}(-1)\1$. Let
$$
C_k(\fg)=V_{\wg}(k,0)\otimes \F
$$
be the tensor product of the vertex operator algebra $V_{\wg}(k,0)$ and the vertex superalgebra $\F$. Then $C_k(\fg)$ is  a vertex algebra.  Define vertex operators $Q_+^{st}(z)$  and $\psi_+(z)$ as follows:
$$
Q_+^{st}(z)=\sum\limits_{n\in\Z}Q_+^{st}(n)z^{-n-1}:=\sum\limits_{\al\in\Delta_+}e_{\al}(z)\psi_{-\al}(z)
-\frac{1}{2}\sum\limits_{\al,\be,\gamma\in\Delta_+}c_{\al,\be}^{\gamma}\psi_{-\al}(z)\psi_{-\be}(z)\psi_{\gamma}(z),
$$
$$
\psi_+(z)=\sum\limits_{n\in\Z}\psi_+(n)z^{-n}:=\sum\limits_{\al\in\Delta_+}\bar{\chi}_+(e_{\al})\psi_{-\al}(z),
$$
where $\bar{\chi}_+$ is defined as above, and $c_{\al,\be}^{\gamma}$ is the structure constant of $\fg$, that is, for $\al,\be\in\Delta_+$, $e_{\al}\in\fg_{\al}$, $e_{\be}\in\fg_{\be}$,
$$
[e_{\al}, e_{\be}]=\sum\limits_{\gamma\in\Delta_+}c_{\al,\be}^{\gamma}e_{\gamma}.
$$
As in \cite{Ar2}, by abuse notations, we set
$$
\begin{array}{ll}
Q_+^{st}:  =Q_+^{st}(0)=\sum\limits_{\al\in\Delta_+,n\in\Z}e_{\al}(-n)\psi_{-\al}(n)
-\frac{1}{2}\sum\limits_{\al,\be,\gamma\in\Delta_+,s+r+m=0}c_{\al,\be}^{\gamma}\psi_{-\al}(s)\psi_{-\be}(r)
\psi_{\gamma}(m),\\
\chi_+: =\chi_+(1)=\sum\limits_{\al\in\Delta_+}\bar{\chi}_+(e_{\al})\psi_{-\al}(1)\\
 Q_+: =Q_{+}^{st}+\chi_{+}.
\end{array}
$$
We have the following lemma \cite{FBZ}, \cite{Ar2}.
\begin{lem}\label{wal-1}
$(Q_+^{st})^2=\chi_{+}^2=[Q_+^{st}, \chi_{+}]=0$,  $Q_+^2=0$.
\end{lem}
Let $\F=\bigoplus_{i\in\Z}\F^i$ be an additional $\Z$-gradation of the vertex algebra $\F$ defined by
$$
\deg \1=0, \ \deg \psi_{\al}(n)=\left\{\begin{array}{rr}
1 \ \ {\rm for} \ \ \al\in\Delta_-\\
-1 \ \ {\rm for} \  \  \al\in\Delta_+.
\end{array}
\right.
$$
For $i\in\Z$, set
$$
C^i_k(\fg)=V_{\wg}(k,0)\otimes \F^i.
$$
This gives a $\Z$-gradation of $C_k(\fg)$:
$$
C_k(\fg)=\bigoplus_{i\in\Z}C^i_k(\fg).
$$
By definition,
$$
Q_+\cdot C^i_k(\fg)\subset C^{i+1}_k(\fg).$$
Then by Lemma \ref{wal-1},  $(C_k(\fg), Q_+)$ is a BRST complex of vertex algebras in the sense of Section 3.15 in \cite{Ar2} ( also see \cite{FBZ}). This complex is called the BRST complex of the quantized Drinfeld-Sokolov (``+``) reduction \cite{FFr4}, \cite{FBZ}, \cite{Ar2}.
The following assertion was proved by B. Feigin and E. Frenkel \cite{FFr4} for generic $k$, by J. de Boer and T. Tjin \cite{dBT} for $k$ in the case that $\fg=\sl_n$, and by E. Frenkel and D. Ben-Zvi for
the general case \cite{FBZ}.
\begin{theorem}\label{coho-1}
The cohomology $H^i(C_k(\fg))$ is zero for all $i\neq 0$.
\end{theorem}
Set
$$
\W^k(\fg):=H^0(C_k(\fg)).
$$
Then $\W^k(\fg)$ is a vertex operator algebra. We have the following result from \cite{FBZ}. 
\begin{theorem}\label{coho-2}
The vertex operator algebra $H^0(C_k(\fg))$ is strongly generated by  elements of degrees $d_i+1$, $i=1,2,\cdots,l$, where $d_i$ is the $i$th exponent of $\fg$ and $l$ is the rank of $\fg$.
\end{theorem}
Denote by $\W_k(\fg)$ the unique simple quotient of $\W^k(\fg)$ at a non-critical level $k$. The following theorem has been proved in \cite{BFM} and \cite{W} in the case that $\fg=\sl_2(\C)$ and
in \cite{DLTV} in the case that $\fg=\sl_3(\C)$ and in \cite{Ar4} for the general case.
\begin{theorem}\label{rationality-1}
The simple $\W$-algebra $\W_k(\fg)$ is rational (and $C_2$-cofinite \cite{Ar3}), and the set of isomorphism classes of minimal series representations of $\W^k(\fg)$ forms  the complete set of the isomorphism classes of simple modules over $\W_k(\fg)$, if $k$ satisfies $k+h^{\vee}=\frac{p}{q}\in{\mathbb Q}_{>0}$, $(p,q)=1$ and
$$
\left\{\begin{array}{l}
p\geq h^{\vee}, \ q\geq h, \ {\rm if} \ (q, r^{\vee})=1\\
p\geq h, \ q\geq r^{\vee}h_{L_{\frak{g}}^{\vee}}, \ {\rm if} \ (q,r^{\vee})=r^{\vee},
\end{array}
\right.
$$
where $L_{\frak{g}}$ is the Langlands dual Lie algebra of $\frak{g}$ and $r^{\vee}$ is the maximal number of the edges in the Dykin diagram of $\frak{g}$.
\end{theorem}

The following conjecture is well known  \cite{FKW}, \cite{BS}, \cite{KWa4}.
\begin{conjecture} Let $\fg$ be a simply-laced simple Lie algebra over $\C$ and $h$ its  Coxeter number. Then for
$p\in{\mathbb Z}$ such that $p\geq h$ and $k=-h+\frac{p+h}{p+h+1}$,
$$
{\W}_k(\frak{g})\cong C_{L_{\widehat{\frak{g}}}(p,0)\otimes L_{\widehat{\frak{g}}}(1,0)}(L_{\widehat{\frak{g}}}(p+1,0)).
$$
\end{conjecture}

\section{Level-Rank Duality}
\setcounter{equation}{0}

For $k\in\Z_{+}$ and  a complex finite-dimensional simple Lie algebra $\frak{g}$ with normalized non-degenerate bilinear form, let $\widehat{\frak g}$ be the corresponding  affine Lie algebra and $L_{\widehat{\frak g}}(k,0)$  the simple vertex operator algebra associated with the integrable highest weight module of $\widehat{\frak g}$ with level $k$. Let $\frak{h}$ be the  Cartan subalgebra of $\frak{g}$ and $L_{\widehat{\frak{h}}}(k,0)$  the associated Heisenberg  vertex operator subalgebra of $L_{\widehat{\frak g}}(k,0)$. Let
$$
K(\frak{g},k)=\{v\in L_{\widehat{\frak g}}(k,0)| [Y(u,z_1),Y(v,z_2)]=0, u\in L_{\widehat{\frak h}}(k,0)\}.
$$
Then  $K({\frak g},k)$ is the so-called parafermion vertex operator algebra (see \cite{BEHHH}, \cite{DLY}, etc.).

Let $s\in\Z_{\geq 2}$ and $\underline{\ell}=(l_1,\cdots,l_s)$ such that $l_{1},\cdots,l_{s}\in \Z_{+}$. Let $L_{\widehat{\frak g}}(l_{i},0)$ be the simple vertex operator algebra associated with the integrable highest weight module of $\widehat{\frak g}$ with level $l_{i}$, $i=1,2,\cdots,s$. Then we have the tensor product vertex operator algebra:
$$V=L_{\widehat{\frak g}}(l_{1},0)\otimes L_{\widehat{\frak g}}(l_{2},0)\otimes \cdots\otimes L_{\widehat{\frak g}}(l_{s},0).$$
Denote
$$
l=|\underline{\ell}|=\sum\limits_{i=1}^{s}l_{i}.
$$
${\frak g}$ can be naturally imbedded into the weight one subspace of $V$ diagonally as follows:
$$
{\frak{g}}\hookrightarrow V_{1}\subseteq V
$$
$$
a\mapsto a(-1){\bf 1}\otimes {\bf 1}\otimes\cdots \otimes {\bf 1}+{\bf 1}\otimes a(-1){\bf 1}\otimes {\bf 1}\otimes \cdots\otimes {\bf 1}+{\bf 1}\otimes \cdots\otimes {\bf 1}\otimes a(-1){\bf 1}.
$$
It is known that the vertex operator subalgebra $U$ of $V$ generated by $\frak{g}$ is isomorphic to the simple vertex operator algebra $L_{\widehat{\frak{g}}}(l,0)$ (\cite{FZ}, \cite{K}, \cite{LL}). Let $C_V(U)$ be the commutant of $U$ in $V$.  We have the following lemma (\cite{JL2}).

\begin{lem}$C_{V}(U)$ is a simple vertex operator subalgebra of $V$.
\end{lem}
Denote
 $$
 s_0=0, \ s_j=l_1+l_2+\cdots+l_j, \ 1\leq j\leq m.
 $$ For $l_k\geq 2$, let $\sl_{l_k}(\C)$ be the simple Lie subalgebra of $\sl_l(\C)$ consisting of matrices $A=(a_{ij})_{l\times l}\in\sl_l(\C)$ such that
 $$
 a_{ij}=0
 $$
 for all the pairs $(i,j)$ such that at least one of $i,j$ is not in the set $\{s_{k-1}+1, s_{k-1}+2,\cdots,s_k\}$.  Let $\h_{\underline{\ell}}$ be the abelian subalgebra of $\sl_l(\C)$ consisting of diagonal matrices $A\in\sl_l(\C)$ such that
 $$
 [A, B]=0,
 $$
 for all $B\in \sl_{l_k}(\C)$ such that $l_k\geq 2$. Then
$$
[\h_{\underline{\ell}}, \bigoplus_{k=1,l_k\geq 2}^{m}\sl_{l_k}(\C)]=0.
$$
Set
$$
\frak{l}_{\underline{\ell}}=\h_{\underline{\ell}}\bigoplus(\bigoplus_{k=1,l_k\geq 2}^{m}\sl_{l_k}(\C)).
$$
Then $\frak{l}_{\underline{\ell}}$ is a Levi subalgebra of $\sl_l(\C)$ and $\frak{h}_{\underline{\ell}}$ is the center of $\frak{l}_{\underline{\ell}}$ which is contained in the (fixed) Cartan subalgebra of $\sl_{l}$. Denote by $L_{\widehat{\frak{l}_{\underline{\ell}}}}(n,0)$ the vertex operator subalgebra of $L_{\widehat{\sl_l}}(n,0)$ generated by $\frak{l}_{\underline{\ell}}$.
It is easy to see that
\begin{equation} \label{levi-tensor}
L_{\widehat{\frak{l}_{\underline{\ell}}}}(n,0) \cong \left(\bigotimes_{k=1,l_k\geq 2}^{m}L_{\widehat{\sl_{l_k}}}(n,0)\right)\bigotimes L_{\widehat{\frak{h}}_{\underline{\ell}}}(n,0),
\end{equation}
 where $L_{\widehat{\frak{h}}_{\underline{\ell}}}(n,0)$ is the Heisenberg vertex operator subalgebra of $L_{\widehat{\sl_{l}}}(n,0)$ generated by $\frak{h}_{\underline{\ell}}$. We denote
$$
K(\sl_l, \frak{l}_{\underline{\ell}},n)=C_{L_{\widehat{\sl_{l}}}(n,0)}(L_{\widehat{\frak{l}_{\underline{\ell}}}}(n,0)).
$$
The following theorem comes from \cite{JL2}.
\begin{theorem}\label{lem4.2} We have
$$
C_{L_{\widehat{\sl_{n}}}(l_1,0)\otimes\cdots\otimes L_{\widehat{\sl_{n}}}(l_m,0)}(L_{\widehat{\sl_{n}}}(l,0))\cong K(\sl_l, \frak{l}_{\underline{\ell}},n).
$$
\[K(\f{sl}_l, \f{l}_{\underline{\ell}}, n)\cong C_{L_{\widehat{\f{sl}_n}}(l_1,0)\otimes \cdots \otimes L_{\widehat{\f{sl}_n}}(l_s,0)}( L_{\widehat{\f{sl}}_n}(l,0))\cong C_{K(\f{sl}_l, n)}(K(\f{sl}_{l_1}, n)\otimes \cdots\otimes K(\f{sl}_{l_s}, n)). \]
\end{theorem}
\begin{remark} If $l_1=\cdots=l_m=1$, then by Theorem \ref{lem4.2},
$$
 C_{L_{\widehat{\frak{sl}_{n}}}(1,0)^{\otimes l}}(L_{\widehat{\sl_{n}}}(l,0))\cong
 K(\sl_{l},n),$$
 which was also established independently by Lam in \cite{L}.
 \end{remark}

\section{The commutant of $L_{\widehat{\sl_{n}}}(l+1,0)$ in $L_{\widehat{\sl_{n}}}(l,0)\otimes L_{\widehat{\sl_{n}}}(1,0)$}
\subsection{}
By the level-rank duality and the reciprocity law, we have
\begin{align}
 C_{L_{\widehat{\f{sl}_n}}(l,0)\otimes  L_{\widehat{\f{sl}_n}}(1,0)}( L_{\widehat{\f{sl}}_n}(l+1,0))
 &\cong
 C_{L_{\widehat{\sl_{l+1}}}(n,0)}(L_{\widehat{\sl_{l}}}(n,0)\otimes L_{\widehat{\frak{h}_{l}}}(n,0)) \nonumber
 \\ &\cong
 C_{K(\f{sl}_{l+1}, n)}(K(\f{sl}_{l}, n)),
 \label{eq:level_rank}
\end{align}
 where $L_{\widehat{\sl_{l}}}(n,0)$ is the vertex operator sualgebra of $L_{\widehat{\sl_{l+1}}}(n,0)$ associated to the simple root system $\{\al_i, 1\leq i\leq l-1\}$ and $L_{\widehat{\frak{h}_{l}}}(n,0)$ is
  the Heisenberg vertex operator subalgebra of $L_{\widehat{\sl_{l+1}}}(n,0)$ generated by $\frak{h}_l=\C(\sum\limits_{i=1}^{l}ih_{\al_i})$.

  Let $\{e_{\al}, f_{\al}, h_{\al_i} \ | \ \al\in\bar{\Delta}_{+}, 1\leq i\leq l\}$ be a Chevalley basis of $\frak{sl}_{l+1}(\C)$,  and let $e_{\al}, f_{\al}, h_{\al}$  be a standard basis of the simple Lie algebra $\frak{sl}_2(\C)$ associated to $\al$, $\al\in\bar{\Delta}_{+}$. For $\al\in\Delta_{+}$, denote by $W^{\al}$ and $\omega^{\al}$ the weight 3 generator and the Virasoro vector  of $K(\f{sl}_{l+1},n)$ associated to $\al$ introduced in \cite{DLY}, respectively. Then
  $$
 \w^{\al}=\frac{1}{2n(n+2)}[-nh_{\al}(-2){\bf 1}-h_{\al}(-1)^2{\bf 1}+2ne_{\al}(-1)f_{\al}(-1){\bf 1}],
 $$
 $$
 \begin{array}{ll}
 W^{\al}=& n^2h_{\al}(-3){\bf 1}+3nh_{\al}(-2)h_{\al}(-1){\bf 1}+2h_{\al}(-1)^3{\bf 1}\\
 & -6nh_{\al}(-1)e_{\al}(-1)f_{\al}(-1){\bf 1}
 +3n^2(e_{\al}(-2)f_{\al}(-1){\bf 1}-e_{\al}(-1)f_{\al}(-2){\bf 1}).
 \end{array}
 $$
For $\al\in\Delta_{+}$, we denote by $W^{4,\al}$ and $W^{5,\al}$ the primary vectors of weight 4 and weight 5 of $K(\f{sl}_{l+1}, n)$ associated to $\al$ introduced in \cite{DLY} respectively. From \cite{DLY}, if $n= 3$, then both $W^{4,\al}$ and $W^{5,\al}$ are zeros for every $\al\in\Delta_{+}$, and $K(\f{sl}_{l+1}, n)$ is generated by $\{W^{\al}, \al\in\Delta_{+}\}$.

 Denote $$V=C_{K(\f{sl}_{l+1}, n)}(K(\f{sl}_{l}, n)).$$ Then $V$ has the conformal vector $\w$  with the central charge $c_V$ as follows:
 $$
 \w=\frac{n+2}{n+l+1}\left(\sum\limits_{1\leq i\leq l}\w^{\al_i+\cdots+\al_l}-
 \sum\limits_{1\leq i\leq j\leq l-1}\frac{1}{(n+l)}\w^{\al_i+\cdots+\al_j}\right),
 $$
$$
 c_V=\frac{l(n-1)(2n+l+1)}{(n+l)(n+l+1)}.
 $$
   For $\al,\be \in\bar{\Delta}$, denote by $c_{\al,\be}\in{\mathbb C}$ the structure constant. That is,
   $$
   [e_{\al}, e_{\be}]=c_{\al,\be}e_{\al+\be},
  \ \  \al,\be,\al+\be\in\bar{\Delta}.$$  It is easy to check that for $\al,\be\in\bar{\Delta}_{+}$ such that $\al+\be\in\bar{\Delta}_{+}$, we have
 $$
 \begin{array}{ll}
 & \omega^{\al}_1W^{\be}\\
 = & \dfrac{1}{n+2}[-3nh_{\be}(-2)h_{\al}(-1)-6h_{\al}(-1)h_{\be}(-1)^2+6n(h_{\al}(-1)-h_{\be}(-1))e_{\be}(-1)f_{\be}(-1)\\
 & +6nh_{\be}(-1)(e_{\al+\be}(-1)f_{\al+\be}(-1)-e_{\be}(-1)f_{\be}(-1))+3n^2(e_{\be}(-2)f_{\be}(-1)-e_{\be}(-1)f_{\be}(-2))\\
 & -3n^2(e_{\al}(-2)f_{\al}(-1)-e_{\al}(-1)f_{\al}(-2)+e_{\al+\be}(-2)f_{\al+\be}(-1)-e_{\al+\be}(-1)f_{\al+\be}(-2))\\
 & +3n^2c_{\al,\be}(f_{\al}(-1)e_{\al+\be}(-1)f_{\be}(-1)+e_{\al}(-1)f_{\al+\be}(-1)e_{\be}(-1))]{\bf 1}
 \end{array}
 $$

 The following lemmas can be checked directly.
 \begin{lem} For $\al,\be\in\bar{\Delta}_{+}$ such that $\al+\be\in\bar{\Delta}$, we have
$$ \omega^{\al}_1W^{\al+\be}=\omega^{\al}_1W^{\be}+\dfrac{1}{n+2}(2W^{\al+\be}+W^{\al}-2W^{\be}),
$$
$$
\omega_1^{\al}W^{\be}+\omega_1^{\be}W^{\al}=\dfrac{1}{n+2}(W^{\al}+W^{\be}-W^{\al+\be}).
$$
 \end{lem}

 \begin{lem}\label{l4.2} For $1\leq p\leq l-1$, we have
 $$
 \begin{array}{ll}
 & e_{\al_p}(0)(\sum\limits_{1\leq i\leq j<l}2(n+2)w_1^{\al_i+\cdots+\al_j}W^{\al_{j+1}+\cdots+\al_l}-\sum\limits_{i=1}^{l}(n+4i-2)W^{\al_i+\cdots+\al_l})\\
 =& (2n^2-12np+16)e_{\al_p}(-3){\bf 1}+(-3n^2+6np-12p-12n-12)h_{\al_p}(-1)e_{\al_p}(-2){\bf 1}\\
 & +(3n^2+6np)h_{\al_p}(-2)e_{\al_p}(-1){\bf 1}+(6n+12p)h_{\al_p}(-1)^2e_{\al_p}(-1){\bf 1}\\
 & -12n\sum\limits_{1\leq i\leq p-1}c_{\al_p,\al_i+\cdots+\al_{p-1}}h_{\al_p}(-1)e_{\al_i+\cdots+\al_p}(-1)f_{\al_i+\cdots+\al_{p-1}}(-1){\bf 1}\\
 & -12n\sum\limits_{1\leq i\leq p}e_{\al_p}(-1)e_{\al_i+\cdots+\al_p}(-1)f_{\al_i+\cdots+\al_p}(-1){\bf 1}\\
 & +12n\sum\limits_{1\leq i\leq p-1}e_{\al_p}(-1)e_{\al_i+\cdots+\al_{p-1}}(-1)f_{\al_i+\cdots+\al_{p-1}}(-1){\bf 1}\\
 & -24\sum\limits_{1\leq i\leq p-1}h_{\al_i+\cdots+\al_{p-1}}(-1)(e_{\al_p}(-2){\bf 1}-h_{\al_p}(-1)e_{\al_p}(-1){\bf 1}).
 \end{array}
 $$
 \end{lem}
 \begin{lem}
 \label{l4.3}
  For $1\leq p\leq l-1$, we have
 $$
 \begin{array}{ll}
 & e_{\al_p}(0)(\sum\limits_{1\leq i\leq j<q\leq l-1}2(n+2)w_1^{\al_i+\cdots+\al_j}W^{\al_{j+1}+\cdots+\al_q}-\sum\limits_{1\leq i\leq q\leq l-1}(n+4i-2)W^{\al_i+\cdots+\al_q})\\
 =& (n+l)[(2n^2-12np+16)e_{\al_p}(-3){\bf 1}+(-3n^2+6np-12p-12n-12)h_{\al_p}(-1)e_{\al_p}(-2){\bf 1}\\
 & +(3n^2+6np)h_{\al_p}(-2)e_{\al_p}(-1){\bf 1}+(6n+12p)h_{\al_p}(-1)^2e_{\al_p}(-1){\bf 1}\\
 & -12n\sum\limits_{1\leq i\leq p-1}c_{\al_p,\al_i+\cdots+\al_{p-1}}h_{\al_p}(-1)e_{\al_i+\cdots+\al_p}(-1)f_{\al_i+\cdots+\al_{p-1}}(-1){\bf 1}\\
 & -12n\sum\limits_{1\leq i\leq p}e_{\al_p}(-1)e_{\al_i+\cdots+\al_p}(-1)f_{\al_i+\cdots+\al_p}(-1){\bf 1}\\
 & +12n\sum\limits_{1\leq i\leq p-1}e_{\al_p}(-1)e_{\al_i+\cdots+\al_{p-1}}(-1)f_{\al_i+\cdots+\al_{p-1}}(-1){\bf 1}\\
 & -24\sum\limits_{1\leq i\leq p-1}h_{\al_i+\cdots+\al_{p-1}}(-1)(e_{\al_p}(-2){\bf 1}-h_{\al_p}(-1)e_{\al_p}(-1){\bf 1})].
 \end{array}
 $$
 \end{lem}
\begin{lem}\label{l4.4}
For $\al,\be,\gamma\in\Delta_{+}$, we have
$$
\w^{\gamma}_2W^{\al}=0, \ \ \omega_2^{\gamma}\omega_1^{\al}W^{\be}=0.
$$
\end{lem}

  Denote
$$
X^{(0)}=0, \ \ X^{(1)}=-(n+2)W^{\al_1},
$$
 $$
 X^{(p)}
 = 2(n+2)\sum\limits_{1\leq i\leq j< q\leq p}\omega_1^{\al_i+\cdots+\al_j}W^{\al_{j+1}+\cdots+\al_q}-\sum\limits_{1\leq i\leq q\leq p}(n+4i-2)W^{\al_i+\cdots+\al_q},
 $$
 for $p\geq 2$.
 Set
 $$
 W=\dfrac{1}{n+l}X^{(l-1)}-\dfrac{1}{n+l+1}X^{(l)}.
 $$

We  have the following result.
 \begin{theorem} For $l\geq 1$,
 $W$  is a weight three generator  of $V=C_{K(\f{sl}_{l+1}, n)}(K(\f{sl}_{l}, n))$. Furthermore, $W$ is a primary element.
 \end{theorem}

 \begin{proof}
 It follows from  Lemmas \ref{l4.2} and \ref{l4.3} that for $1\leq p\leq l-1$,
 \begin{equation}\label{e4.1}
 e_{\al_p}(0)W=0.
 \end{equation}
 Consider  $L_{\widehat{\sl_{l+1}}}(n,0)$ as a module of the simple Lie algebra generated by $e_{\al_p}, f_{\al_p},h_{\al_p}$, for $1\leq p\leq l-1$. Then  (\ref{e4.1}) implies that $W$ is a highest weight vector.
 Since  $h_{\al_p}(0)W=0$ for $1\leq p\leq l$, we have $f_{\al_p}(0)W=0$, $1\leq p\leq l-1$. Notice that $W\in K(\f{sl}_{l+1}, n)$. It follows that $W\in C_{K(\f{sl}_{l+1}, n)}(K(\f{sl}_{l}, n))$. Lemma \ref{l4.4} implies that $W$ is a primary element.
 \end{proof}

For $\al\in\Delta_{+}$, from \cite{DLY} we have
\begin{lem}\label{l4.5}
$$
W^{\al}_5W^{\al}=12n^3(n-2)(n-1)(3n+4){\bf 1};
$$
$$
W^{\al}_5W^{\al}=36n^3(n-2)(n+2)(3n+4)\omega^{\al};
$$
$$
W^{\al}_2W^{\al}=18n^3(n-2)(n+2)(3n+4)\omega^{\al}_{-2}{\bf 1}.
$$
 \end{lem}
 The following lemmas  can be checked directly.
\begin{lem}\label{l4.6}
For $\al,\be\in\Delta_{+}$ such that $\al\neq \pm\be$
 $$
 W^{\al}_{5}W^{\be}=6(\al|\be)n^3(n-1)(n-2);
 $$
 $$
 W^{\al}_2W^{\be}=18(\al|\be)n^3(n-2)(n+2)^2\omega_0^{\al}\omega^{\be};
 $$
 $$
 W^{\al}_3W^{\be}=\left\{\begin{array}{lr}-18n^3(n-2)(n+2)(\omega^{\al}+\omega^{\be}-\omega^{\al+\be}) \ {\rm if} \ (\al|\be)=-1,\\
 18n^3(n-2)(n+2)(\omega^{\al}+\omega^{\be}-\omega^{\al-\be}) \ \ \ {\rm if} \ (\al|\be)=1,\\
 0 \ \ \ {\rm if} (\al|\be)=0.
\end{array}\right.
 $$
\end{lem}
\begin{lem}\label{l4.6}
For $\al,\be,\gamma\in\Delta_{+}$ such that $(\al|\be)=-1$, we have
$$
W^{\al}_3\omega^{\al}_1W^{\be}=18n^3(n-2)[-(n+4)\omega^{\be}+(n+4)\omega^{\al+\be}-3(n+2)\omega^{\al}];
$$
$$
W^{\be}_3\omega^{\al}_1W^{\be}=18n^3(n-2)[(n+4)\omega^{\al}+3(3n+4)\omega^{\be}-(n+4)\omega^{\al+\be}];
$$
$$
W^{\al+\be}_3\omega^{\al}_1W^{\be}=18n^3(n-2)(n+1)(\omega^{\al}-\omega^{\be}-3\omega^{\al+\be});
$$
$$
W^{\gamma}_3\omega^{\al}_1W^{\be}=18n^3(n-2)(\omega^{\al+\be-\gamma}+\omega^{\be}-\omega^{\al+\be}-\omega^{\be-\gamma}), \ \ {\rm if} \ (\gamma|\al)=0, \ (\gamma|\be)=1;
$$
$$
W^{\gamma}_3\omega^{\al}_1W^{\be}=18n^3(n-2)(\omega^{\al+\be}+\omega^{\be+\gamma}-\omega^{\al+\be+\gamma}-\omega^{\be}), \ \ {\rm if} \ (\gamma|\al)=0, \ (\gamma|\be)=-1;
$$
$$
W^{\gamma}_3\omega^{\al}_1W^{\be}=18n^3(n-2)(\omega^{\al}+2\omega^{\be}+3\omega^{\gamma}-\omega^{\al+\gamma}-2\omega^{\gamma-\be}), \ \ {\rm if} \ (\gamma|\al)=-1, \ (\gamma|\be)=1;
$$
$$
W^{\gamma}_3\omega^{\al}_1W^{\be}=0, \ \ {\rm if} \ (\al|\gamma)=1, \ (\gamma|\be)=0.
$$
\end{lem}

By Lemmas \ref{l4.5}-\ref{l4.6}, we can deduce  the following lemma.
\begin{lem}\label{l4.7}
$$W_5W=\dfrac{6n^3l(n-1)(n-2)(n+2l)(2n+l+1)(3n+2l+2)}{(n+l+1)(n+l)};
$$
$$
W_3W=36n^3(n-2)(n+2l)(3n+2l+2)\omega.
$$
\end{lem}
Notice that
$$
W_2W=-W_2W+\sum\limits_{j=1}^{\infty}\frac{(-1)^{j+1}}{j!}L(-1)^{j}W_{j+2}W.
$$
So by Lemma \ref{l4.7}, we have
\begin{equation}\label{We1}
W_2W\in L(c_V,0).
\end{equation}

\vskip 0.3cm
For $k+n=\frac{n+l}{n+l+1}$, let $\wv={\mathcal W}_{k}(\sl_n)$ be the  ${\mathcal W}$-algebra associated to $\sl_n$. Denote by $\widetilde{\omega}$ its conformal vector. Let $\wW$ be the weight three primary vector of  ${\mathcal W}_{k}(\sl_n)$  such that
\begin{equation}\label{eqb4.1}
(\wW,\wW)=-\dfrac{6n^3l(n-1)(n-2)(n+2l)(2n+l+1)(3n+2l+2)}{(n+l+1)(n+l)}.
\end{equation}
Then by the fusion rules of the Virasoro algebra, we have
\begin{equation}\label{e5.9}
\wW_3\wW=36n^3(n-2)(n+2l)(3n+2l+2)\ww.
\end{equation}
Note that $c_V=c_{\wv}$. So $L(c_V,0)=L(c_{\wv},0)$. We denote $c=c_V=c_{\wv}$.  We may identify the Virasoro vertex operator subalgebra $L(c,0)$ in $V$ and $\wv$.

In view of \eqref{eq:level_rank},
the character of $V$ has been given
in \cite{KWa2}.
This coincides with the
 character of $\wv$ that was conjectured in \cite{FKW} and proved in \cite{Ar2}.
Therefore we have the following.
\begin{lem}\label{l4.11}
$V$ and $\wv$ share the same  characters.
\end{lem}
  It is known that $\mw_{k}(\sl_n)$ is strongly generated by  $n-1$ quasi-primary vectors with  weights $2,3,\cdots,n$
\cite{FBZ}.
 For $n=3$, we have the following lemma.
\begin{lem}\label{l5.11}
If $n=3$, then $W_iW\in L(c,0)$, for $i\geq 0$.
\end{lem}
 \begin{proof} By Lemma \ref{l4.7} and (\ref{We1}), $W_{i}W\in L(c_V,0)$, $i\geq 2$.
 Then by the skew-symmetry, we only need to show that $W_1W\in L(c_V,0)$. By Lemma \ref{l4.11} and  \cite{FBZ} (Chapter 14),  for $n=3$, $V$ has no primary vector of weight 4. So we may assume that
 $$
 W_1W=aL(-1)W+u,
 $$
 for some $u\in L(c,0)$ and $a\in\C$. If $a\neq 0$, then
\begin{equation}\label{We2}
L(1)(W_1W-u)=aL(1)L(-1)W=6aW.
\end{equation}
On the other hand, we have
$$
L(1)W_1W=\sum\limits_{i=0}^{2}\left(\begin{array}{c}2\\i\end{array}\right)(\w_iW)_{3-i}W=3W_2W\in L(c,0),
$$
contradicting (\ref{We2}). We deduce that $a=0$, which implies that $W_1W\in L(c,0)$.
 \end{proof}
 Denote by $U$ the vertex operator subalgebra of  $V$   generated by $W$. Then by Lemma \ref{l4.11}, for $n=3$, $U$ is linearly spanned by
$$
\begin{array}{ll}
{\mathcal S}=& \{L(-m_{s})\cdots L(-m_{1})W_{-k_{p}}\cdots W_{-k_{1}}{\bf 1},
 k_{p}\geq \cdots\geq
k_{1}\geq 1, \\
& \ n\in\Z, \ m_{s}\geq\cdots \geq m_{1}\geq 1, \ s,p\geq
0\}.
\end{array}$$
From \cite{FBZ} (Chapter 14),  for $n=3$, $\wv=\mw_k(\sl_n)$ is linearly spanned by
$$
\begin{array}{ll}
\widetilde{\mathcal S}=& \{L(-m_{s})\cdots L(-m_{1})\wW_{-k_{p}}\cdots \wW_{-k_{1}}{\bf 1}, \ k_{p}\geq \cdots\geq
k_{1}\geq 1, \\
& \ n\in\Z, \ m_{s}\geq\cdots \geq m_{1}\geq 1, \ s,p\geq
0\}.
\end{array}$$

\vskip 0.3cm
For any
$$
u=\sum\limits_{j=1}^qb_{j}L(-n_{j1})\cdots L(-n_{jt_{j}})W_{-r_{j1}}
\cdots W_{-r_{jp_{j}}}{\bf 1}, $$
where $q\geq 1, t_{j},p_{j}\geq 0$, $n_{j1},\cdots,n_{jt_{j}}, r_{j1},\cdots, r_{jp_{j}}\in\Z_{+}$,   $j=1,2,\cdots,q$,
we always denote
$$\widetilde{u}=\sum\limits_{j=1}^qb_{j}L(-n_{j1})\cdots L(-n_{jt_{j}})\wW_{-r_{j1}}
\cdots \wW_{-r_{jp_{j}}}{\bf 1}.$$
We have the following lemma.
\begin{lem}\label{l5.10} For $m\in\Z$ and
$$
u=\sum\limits_{j=1}^qb_{j}L(-n_{j1})\cdots L(-n_{jt_{j}})W_{-r_{j1}}
\cdots W_{-r_{jp_{j}}}{\bf 1}, $$
where $q\geq 1, t_{j},p_{j}\geq 0$, $n_{j1},\cdots,n_{jt_{j}}, r_{j1},\cdots, r_{jp_{j}}\in\Z_{+}$,   $j=1,2,\cdots,q$, if
$$
W_mu=v
$$
is a linear combination of vectors from ${\mathcal S}$, then
$$
\wW_m\widetilde{u}=\widetilde{v}.
$$
\end{lem}
\begin{proof}
We may assume that $u$ is homogeneous. We prove the lemma by induction on the weight of $u$. If $\wt u\leq 3$, the lemma is true by the fact that $(W,W)=(\wW,\wW)$ and the fusion rules of the Virasoro algebra $L(c,0)$. Suppose that the lemma holds for homogeneous $u$ such that $\wt u\leq N$. We now assume that $\wt u=N$. For each monomial $u_j=L(-n_{j1})\cdots L(-n_{jt_{j}})W_{-r_{j1}}
\cdots W_{-r_{jp_{j}}}{\bf 1}$, let $W_mu_j=v_j$ be a linear combination of elements in ${\mathcal S}$. It is obvious that we may assume that $t_j=0$. If $m\leq -r_{j1}$, then we have
$
\wW_m\widetilde{u}_j=\widetilde{v}_j.
$ If $m> -r_{j1}$, we have
$$
\begin{array}{ll}
W_mu_j=W_mW_{-r_{j1}}\cdots W_{-r_{jp_{j}}}{\bf 1}\\
=\sum\limits_{s=0}^{\infty}\left(\begin{array}{c}m\\s\end{array}\right)(W_sW)_{m-r_{j1}-s}W_{-r_{j2}}\cdots W_{-r_{jp_j}}{\bf 1}
+W_{-r_{j1}}W_mW_{-r_{j2}}\cdots W_{-r_{jp_{j}}}{\bf 1}
\end{array}
$$
$$
\begin{array}{ll}
\wW_mu_j=\wW_m\wW_{-r_{j1}}\cdots \wW_{-r_{jp_{j}}}{\bf 1}\\
=\sum\limits_{s=0}^{\infty}\left(\begin{array}{c}m\\s\end{array}\right)(\wW_s\wW)_{m-r_{j1}-s}\wW_{-r_{j2}}\cdots \wW_{-r_{jp_j}}{\bf 1}
+\wW_{-r_{j1}}\wW_m\wW_{-r_{j2}}\cdots \wW_{-r_{jp_{j}}}{\bf 1}%
\end{array}
$$
Then the lemma follows from Lemma \ref{l5.11} and the inductive assumption.
\end{proof}
\begin{lem}\label{l5.8} Suppose that  $n=3$. For any $q\geq 1, t_{j},p_{j}\geq 0$, $n_{j1},\cdots,n_{jt_{j}}, r_{j1},\cdots, r_{jp_{j}}\in\Z_{+}$,   $j=1,2,\cdots,q$,   if
$$
u=\sum\limits_{j=1}^qb_{j}L(-n_{j1})\cdots L(-n_{jt_{j}})W_{-r_{j1}}
\cdots W_{-r_{jp_{j}}}{\bf 1}=0
$$
for some $b_j\in\C$,
 then
$$
\widetilde{u}=\sum\limits_{j=1}^qb_{j}L(-n_{j1})\cdots L(-n_{jt_{j}})\wW_{-r_{j1}}
\cdots \wW_{-r_{jp_{j}}}{\bf 1}=0.
$$
\end{lem}
\begin{proof}
We may assume  that $u$ is a linear combination of homogeneous elements having same weight. Suppose that $\tilde{u}\ne 0.$ Since $\mw_k(\sl_n)$ is self-dual and generated by $\wW$,
there is $\wW_{r_{1}}\wW_{r_{2}}\cdots \wW_{r_{q}}{\bf 1}\in \mw_k(\sl_n)$ such that
\begin{equation}\label{bili-3}
(\widetilde{u}, \wW_{r_{1}}\wW_{r_{2}}\cdots \wW_{r_{q}}{\bf 1})\neq 0.
\end{equation}
{\bf Claim} \ For any $k_{p}\geq \cdots\geq
k_{1}\geq 1$, $q_1,q_2,\cdots,q_t\in\Z$, $m_{s}\geq\cdots \geq m_{1}\geq 1$, $p,t\geq
0$,
\begin{eqnarray}
& &(L(-m_{s})\cdots L(-m_{1})W_{-k_{p}}\cdots W_{-k_{1}}{\bf 1}, W_{q_1}W_{q_2}\cdots W_{q_t}{\bf 1})\nonumber\\
& &\ \ \ \ =(L(-m_{s})\cdots L(-m_{1})\wW_{-k_{p}}\cdots \wW_{-k_{1}}{\bf 1}, \wW_{q_1}\wW_{q_2}\cdots \wW_{q_t}{\bf 1}).\label{bili-2}
\end{eqnarray}
We may assume that
$$\wt(L(-m_{s})\cdots L(-m_{1})\wW_{-k_{p}}\cdots \wW_{-k_{1}}{\bf 1})=\wt( \wW_{q_1}\wW_{q_2}\cdots \wW_{q_t}{\bf 1}).$$
We prove (\ref{bili-2}) by induction on $\wt(L(-m_{s})\cdots L(-m_{1})\wW_{-k_{p}}\cdots \wW_{-k_{1}}{\bf 1})$. By Lemma \ref{l4.7} and (\ref{eqb4.1}),  (\ref{bili-2}) holds if $\wt(L(-m_{s})\cdots L(-m_{1})\wW_{-k_{p}}\cdots \wW_{-k_{1}}{\bf 1})\leq 3$. Now assume that the claim holds for $\wt(L(-m_{s})\cdots L(-m_{1})\wW_{-k_{p}}\cdots \wW_{-k_{1}}{\bf 1})< N$. Then by Lemma \ref{l5.11} and the inductive assumption, the claim holds for $L(-m_{s})\cdots L(-m_{1})\wW_{-k_{p}}\cdots \wW_{-k_{1}}{\bf 1}$ such that $\wt(L(-m_{s})\cdots L(-m_{1})\wW_{-k_{p}}\cdots \wW_{-k_{1}}{\bf 1})= N$.

By the claim and (\ref{bili-3}), we have
$$
(u, W_{q_1}W_{q_2}\cdots W_{q_t}{\bf 1})\neq 0,
$$
which contradicts the assumption that $u=0$.
\end{proof}
\begin{theorem}
For  $n=3$ and $k=-n+\frac{n+l}{n+l+1}$, we have $C_{K(\f{sl}_{l+1}, n)}(K(\f{sl}_{l}, n))\cong \W_k(\sl_n)$.
\end{theorem}
\begin{proof}
Define $\varphi: U\rightarrow \wv$ as follows: for any
$$
u=\sum\limits_{j=1}^qb_{j}L(-n_{j1})\cdots L(-n_{jt_{j}})W_{-r_{j1}}
\cdots W_{-r_{jp_{j}}}{\bf 1}, $$
where $q\geq 1, t_{j},p_{j}\geq 0$, $n_{j1},\cdots,n_{jt_{j}}, r_{j1},\cdots, r_{jp_{j}}\in\Z_{+}$,   $j=1,2,\cdots,q$,
$$\varphi(u)=\widetilde{u}=\sum\limits_{j=1}^qb_{j}L(-n_{j1})\cdots L(-n_{jt_{j}})\wW_{-r_{j1}}
\cdots \wW_{-r_{jp_{j}}}{\bf 1}.
$$
By Lemma \ref{l5.8} and Lemma \ref{l5.11}, $\varphi$ ia a surjective vertex operator algebra homomorphism from $U$ to $\wv$. Since $U\subseteq V$, and $V$ and $\wv$ share the same character, we deduce that $U=V$ and $\varphi$ is an isomorphism.
\end{proof}

\end{document}